\documentclass[a4paper,10pt,reqno]{amsart}
\usepackage[applemac]{inputenc}
\usepackage[T1]{fontenc}
\usepackage{amsthm}
\usepackage{latexsym,amsmath,amsfonts,amscd,amssymb}
\usepackage{graphicx}
\usepackage{xypic}
\usepackage{array,booktabs,arydshln,xcolor}
\usepackage{setspace}
\usepackage{fancyhdr}
\usepackage{fancybox}
\usepackage{multicol}
\usepackage[all]{xy}
\newtheorem{Theorem}{Theorem}[section]
\newtheorem{Lemma}[Theorem]{Lemma}
\newtheorem{Proposition}[Theorem]{Proposition}

\newtheorem{Definition}[Theorem]{Definition}

\theoremstyle{definition}

\newtheorem{Remark}[Theorem]{Remark}

\newcommand{\KK}{{\mathbb K}}

\newcommand{\PP}{{\mathbb P}}
\newcommand{\GG}{{\mathbb G}}

\newcommand{\HH}{{\mathcal H}}
\newcommand{\UU}{{\mathcal U}}
\newcommand{\QQ}{{\mathcal Q}}

\newcommand{\A}{{\mathcal A}}
\newcommand{\BB}{{\mathcal B}}
\newcommand{\OO}{{\mathcal O}}

\newcommand{\rk}{\operatorname{rk}}
\newcommand{\Hom}{\operatorname{Hom}}
\newcommand{\End}{\operatorname{End}}

\newcommand{\im}{\operatorname{Im}}

\newcommand{\Seg}{\operatorname{Seg}}
\makeindex

\author[E. Arrondo]{Enrique Arrondo}
\address{Universidad Complutense de Madrid, Fac. de CC. Matemáticas, Plaza de Ciencias 3, 28040, Madrid, Spain}
\email{arrondo@mat.ucm.es}
\author[S. Marchesi]{Simone Marchesi}
\address{Universidad Complutense de Madrid, Fac. de CC. Matemáticas, Plaza de Ciencias 3, 28040, Madrid, Spain\\
Università degli Studi di Milano, Dipartimento di Matematica, via Saldini 50, 20133, Milano, Italy}
\email{smarchesi@mat.ucm.es, simone.marchesi@unimi.it}
\subjclass[]{}
\thanks{The authors thank the Spanish Ministry of Education for funding the research project MTM2009-06964, in whose framework this work has been developed.}

\begin{document}
\title{Jumping pairs of Steiner bundles}
\date{}
\maketitle
\begin{abstract}
In this work we introduce the definition of Schwarzenberger bundle on a Grassmannian. Recalling the notion of Steiner bundle, we generalize the concept of jumping pair for a Steiner bundle on a Grassmannian. After studying the jumping locus variety and bounding its dimension, we give a complete classification of Steiner bundle with jumping locus of maximal dimension, which all are Schwarzenberger bundles.
\end{abstract}
\section*{Introduction} 
Steiner bundles over projective spaces were first defined in 1993 (see \cite{dolg-kap}) by Dolgachev and Kapranov as particular vector bundles of rank $n$ in $\PP^n$, although the first particular case was introduced in 1961 (see \cite{schw}) by Schwarzenberger. In 2000 (see \cite{val1}), Vall\`{e}s proves a result that characterizes when a Steiner bundle $F$ can be described also as a Schwarzenberger, studying a particular family of hyperplanes $\{H_i\}$, which are called \emph{unstable hyperplanes}, a technique also used in 2001(see \cite{anc-ott}) by Ancona and Ottaviani. In parallel to a paper of Vall\`{e}s (see \cite{val2}) in this direction, finally in \cite{arr} there is a full generalization to higher rank of the notion of Schwarzenberger bundle and the one of unstable hyperplane for a Steiner bundle $F$, now called \emph{jumping hyperplane}. The main result is that Steiner bundles whose jumping locus has maximal dimension can be classified and are Schwarzenberger bundles. In this paper we follow the suggestion in \cite{arr} for the notion of Steiner bundle on Grassmannians, which fits in the general definition  by Mir\'{o}-Roig and Soares for an arbitrary variety (see \cite{soar-miro}). Then we recover all the results of that paper in this new situation.

In Section \ref{sec-Steiner}, we will state the general definition of a Steiner bundle for Grassmannians (see Definition \ref{def-Steiner}). 
After showing a geometrical interpretation of the definition, we will then give a lower bound for the possible ranks of the bundles just defined (see Theorem \ref{thm-rank}).

In Section \ref{sec-Schw} we will give the definition of Schwarzenberger bundle (see Definition \ref{def-Schw}), generalizing the one given in \cite{arr}.

In Section 3 we will then give the definition of jumping pair for a Steiner bundle and also give an algebraic structure to the set of all of them (see Definition \ref{def-jump}). We will bound the dimension of its locus through the description of its tangent space, which will give us information of the jumping locus as a projective variety (see Theorem \ref{thm-tang}).

Finally, in Section 4, we will classify the Steiner bundles whose jumping locus has maximal dimension and we will get that also in this case they are Schwarzenberger bundles (see Theorem \ref{thm-class}).

\section{Steiner bundles on $\GG(k,n)$}\label{sec-Steiner}
In this section we will set the notation we will use throughout this work and we will recall the definition of a Steiner bundle.\\
Let $\KK$ be an algebraically closed field with $\mbox{char} \:\KK = 0$ and let $V$ be a vector space over $\KK$. Let us construct the projective space $\PP^n = \PP(V)$ as the equivalence classes of the hyperplanes of $V$, or equivalently the equivalence classes of the lines of the dual vector space $V^*$. In this setting we define the projective grassmannian as
$$
\GG(k,\PP(V)) = \GG(k,n),
$$
or, vectorially, as $G(k+1,V^*)$, i.e. the set of the $(k+1)$-linear subspaces of $V^*$.
\begin{Definition}\label{def-Steiner}
Let $S,T$ be two vector spaces over $\KK$, respectively $s$ and $t$-dimensional.\\
We will call an $(s,t)$-Steiner bundle, over $\GG(k,n)$, the vector bundle defined by the resolution 
$$
0 \longrightarrow S \otimes \UU \longrightarrow T \otimes \OO_{\GG} \longrightarrow F \longrightarrow 0,
$$
where $\UU \rightarrow \GG(k,n)$ denotes the universal bundle of rank $k+1$ on the grassmannian and $\OO_\GG \rightarrow \GG(k,n)$ denotes the trivial line bundle.
\end{Definition}
First of all notice that the rank of the bundle we have constructed is $\rk F = t - s(k+1)$. Furthermore, observe that the given definition agrees with the following one proposed by Miró-Roig and Soares, because the pair of vector bundles $(\UU,\OO_\GG)$ is strongly exceptional. 
\begin{Definition}[\cite{soar-miro}]
A vector bundle $F$ on a smooth irreducible algebraic variety $X$ is called a Steiner vector bundle if it is defined by an exact sequence of the form
$$
0 \longrightarrow E_0^s \longrightarrow E_1^t \longrightarrow E \longrightarrow 0,
$$
where $s,t \geq 1$ and $(E_0, E_1)$ is an ordered pair of vector bundles satisfying the two following conditions:
\begin{description}
\item[(i)] $(E_0,E_1)$ is strongly exceptional;
\item[(ii)] $E_0^\lor \otimes E_1$ is generated by its global sections.
\end{description}
\end{Definition}
In order to define a Steiner bundle on $\GG(k,n)$, we could have also used the strongly exceptional pair $(\OO_\GG,\QQ)$; indeed, the two possible definitions are strongly related as we will see in the proof of Theorem \ref{teo-classtriv}.\\
We wonder which are the possible ranks for a Steiner bundle on $\GG(k,n)$, because such problem was solved for projective spaces by Dolgachev and Kapranov in \cite{dolg-kap}, where they prove that the rank of a non trivial Steiner bundle over $\PP^n$ is at least $n$.
\begin{Theorem}\label{thm-rank}
Let $F$ be a Steiner bundle over $\GG(k,n)$; then it will have rank
$$\rk F \geq \min((k+1)(n-k), (n-k)\dim S).$$
\end{Theorem}
\begin{proof}
Let us consider the morphism
\begin{equation} \label{eq-def}
S \otimes \UU \longrightarrow T \otimes \OO_\GG \dashrightarrow F,
\end{equation}
and let us write $r = \rk F$. We already know that $r = \dim T - (k+1)\dim S$. 
We want to study for which points of the Grassmannian such morphism is not injective and in order to do so we will use Porteous' formula. The expected codimension of such set, given by the formula, is equal to $r+1$, so if we have $\dim\GG(k,n) <r+1$ then there will be no points that drop the rank of the morphism and we can ensure its injectivity. In the case the set is not empty, let us compute the fundamental class obtained applying Porteous: if the class is empty, we will still ensure injectivity and thus have a Steiner bundle, otherwise, we will not.\\
The fundamental class can be computed as
$$
\Delta_{r+1,1}\left(c_y\left(T \otimes \OO_\GG - S \otimes \UU\right)\right) = \Delta_{r+1,1}\left(c_y\left(\frac{1}{S\otimes \UU}\right)\right) = \Delta_{r+1,1}\left(c_y\left( S\otimes \QQ \right)\right) ,
$$
where $c_y$ denotes the Chern polynomial of the bundle and, following the notation of \cite{acgh}, the symbol $\Delta_{r+1,1}$ just denotes the coefficient of the term of degree $r+1$ in the polynomial.\\
We thus need to describe the coefficients of
$$
c_y(\QQ)^{\dim S^*} = \left( 1 + \sigma_1 y + \sigma_2 y^2 + \ldots + \sigma_{n-k}y^{n-k} \right)^{\dim S^*},
$$
where each Chern class $c_i(\QQ)$ is equal to $\sigma_i$, the special Schubert cycle of codimension $i$ for the Grassmannian $\GG(k,n)$. Let us observe that, having a power of the polynomial, we need to know how the Schubert cycles intersect among each other and we know that, if we do not exceed the dimension of the Grassmannian, such intersection is always non empty, see \cite{mar-tesi}. So, if we do not exceed the dimension of the Grassmannian in the Chow ring, the intersection of special Schubert cycles is always non empty and we prove the theorem.
\end{proof}
\begin{Remark}\label{intStein}
We generalize the geometrical interpretation written in \cite{schw} and \cite{arr}. An injective morphism $\UU \rightarrow T \otimes \OO_\GG$ is equivalent to fix a $(n+1)$-codimensional space in $\Lambda \subset \PP(T)$. Notice that all hyperplanes of $\PP(T)$ which contain $\Lambda$ are in a one-to-one correspondence with the points of a linear subspace of maximal dimension, disjoint from $\Lambda$ and that generates with $\Lambda$ all $\PP(T)$. Due to the codimension we can build a one-to-one correspondence between such hyperplanes and the ambient space $\PP^n$. In the same way each injective morphism $S \otimes \UU \rightarrow T \otimes \OO_\GG$ defines $s$ subspaces $\Lambda_1,\ldots,\Lambda_s\subset \PP(T)$ of codimension $n+1$. We are now able to build several correspondences, one for each $i = 1,\ldots,s$, in order to have
$$
\PP(T)^*_{\Lambda_i} := \left\{ \mbox{hyperplanes of} \:\: \PP(T) \:\: \mbox{that contain} \:\: \Lambda_i\right\} \simeq \PP^n, \:\mbox{for each}\:i = 1,\ldots, s.
$$
Let us fix now a point $\Gamma \in \GG(k,n)$, which gives us a subspace of codimension $k+1$ in $\PP(T)$, one for each $i$, that contains $\Lambda_i$. We have thus obtained the correspondences
$$
A_i = \left\{ \mbox{subspaces of codimension} \:\: k+1 \:\: \mbox{in} \:\: \PP(T) \:\: \mbox{that contains} \:\: \Lambda_i \right\} \simeq \GG(k,n), \:\mbox{for every}\: i = 1,\ldots,s.
$$
The projectivization of the fiber $F_\Gamma$ is given by the intersection of the chosen $s$ subspaces. The projectivization of the bundle is given by the union of all $\PP(F_\Gamma)$, for each $\Gamma \in \GG(k,n)$.
\end{Remark}
We will now see an extremely important result, because it gives us an equivalent definition for a Steiner bundle in terms of linear algebra. From now on, when we consider a Steiner bundle, we will refer to any of the two proposed definitions.
\begin{Lemma} \label{lem-def}
Given $S$ and $T$ two vector spaces over $\KK$, the followings are equivalent:
\begin{description}
\item[(i)] a Steiner bundle $F$ given by the resolution
$$
0\longrightarrow S \otimes \UU \longrightarrow T \otimes \OO_{\GG} \longrightarrow F \longrightarrow 0
$$
\item[(ii)] a linear application
$$
T^* \stackrel{\varphi}{\longrightarrow} S^* \otimes H^0(\UU^\lor) = \Hom(H^0(\UU^\lor)^*,S^*)
$$
such that, for every $u_1,\ldots,u_{k+1} \in H^0(\UU^\lor)^*$ linearly independent and for every $v_1,\ldots,v_{k+1} \in S^*$, there exists an $f \in \Hom(H^0(\UU^\lor)^*,S^*)$ such that $f \in \im \varphi$ and $f(u_j) = v_j$ for each $j=1,\ldots,k+1$.
\end{description}
\end{Lemma}
\begin{proof}
Let us consider the map
$$
S \otimes \UU \longrightarrow T \otimes \OO_{\GG}
$$
and its dual one
$$
\psi : T^* \otimes \OO_\GG \longrightarrow S^* \otimes \UU^\lor = \mathcal{H}{om}(\UU,S^*\otimes \OO_\GG)
$$
that gives us the induced map on the global sections
$$
\varphi : T^* \longrightarrow S^* \otimes H^0(\UU^\lor) = \Hom(H^0(\UU^\lor)^*,S^*).
$$
We need to characterize $\varphi$ in order to have the map $\psi$ surjective, that is equivalent to ask for $\psi$ to be surjective in each fiber.\\
A point $\Gamma \in \GG(k,n)$ is in correspondence with $k+1$ independent vectors $u_1,\ldots,u_{k+1} \in H^0(\UU^\lor)^*$, so that the bundle morphism in the fiber associated to $\Gamma$ corresponds to the restriction of $\varphi$ of the type
$$
\tilde{\varphi} :  T^* \longrightarrow \Hom(<u_1,\ldots,u_k+1>,S^*).
$$
The requested characterization will be exactly the one stated in point (ii), because $\tilde{\varphi}$ is surjective for every fiber if and only if for every $k+1$ independent vectors of $H^0(\UU^\lor)^*$, related with the fixed point in the Grassmannian, and $k+1$ vectors of $S^*$, that can be chosen as image of the previous set, there exists $f \in \Hom(H^0(\UU^\lor)^*,S^*)$ with $f \in \im \varphi$ such that $f(u_j)=v_j$ for every $j$ from 1 to $k+1$. 
\end{proof}
\subsection{Steiner bundles and linear algebra} We prove some general results of linear algebra that we will discover to be very useful in the specific setting of Steiner bundles.\\
Let $U$ and $V$ be two vector spaces, of dimension respectively $r$ and $s$.\\
Let $W \subset \Hom(U,V)$ be the vector space characterized by the following property, which we will denote by $P_k$: for every $k+1$ independent vectors $\tilde{u}_1,\ldots,\tilde{u}_{k+1} \in U$ and for every $\tilde{v}_1,\ldots,\tilde{v}_{k+1} \in V$ there exists an element $f \in W$ such that $f(\tilde{u}_j) = \tilde{v}_j$, for every $j = 1,\ldots,k+1$. This is equivalent to ask that for every vector subspace $U'\subset U$, with $\dim U' = k+1$, we have the following diagram
$$
\xymatrix{
W \ar[dr]^{\alpha} \ar@{^{(}->}[d] \\
\Hom(U,V) \ar@{>>}[r] & \Hom(U',V)
}
$$
where the induced map $\alpha$ will always be surjective.\\
It is immediate to prove the following result
\begin{Lemma}\label{lem-tec3}
Let $W \subset\Hom(U,V)$ be a vector subspace that satisfies the property $P_k$. If $\dim V = k+1$, then the map $W \longrightarrow Hom(U,V)$ is also surjective.
\end{Lemma}
\begin{proof}
Consider an element in $U^* \otimes V$, that we can also write as a linear combination $\sum \lambda_{i,j} u^*_i \otimes v_j$, given by the choice of a basis for each of the vector spaces $U^*$ and $V$. Collect all the elements of the combination using the basis of $V$, in order to write the combination as
$$
\tilde{u_1}^*\otimes v_1 + \ldots + \tilde{u}^*_{k+1} \otimes v_{k+1}.
$$
Notice that having at most $k+1$ independent $\tilde{u_i}^*$ and completing a basis of $U^*$ with independent vectors that must vanish because the rank of the chosen element is determined, we can see the previous sum as an element of $W$.
\end{proof}
We can apply the previous lemma to prove the following result
\begin{Lemma}\label{lem-dual}
Let $W \subset\Hom(U,V)$ be a vector subspace that satisfies the property $P_k$. Then, considering the induced map
$$
W \longrightarrow \Hom(V^*,U^*),
$$
we have that for every $k+1$ independent vectors $\tilde{v}^*_1,\ldots,\tilde{v}^*_{k+1} \in V^*$ and for every $\tilde{u}^*_1,\ldots,\tilde{u}^*_{k+1} \in U^*$ there exists an element $\bar{f} \in W$ such that $\bar{f}(\tilde{v}^*_j) = \tilde{u}^*_j$, for every $j = 1,\ldots,k+1$. We say that $W$ also satisfy the reciprocal of the property $P_k$.
\end{Lemma}
\begin{proof}
Let us consider an arbitrary $(k+1)$-dimensional quotient $Q$ of the vector space $V$, so we are able to construct the following commutative diagram
$$
\xymatrix{
W \ar[r]^(.3){\varphi} \ar[rd]_{\varphi'} \ar@/^3pc/[rr]^f \ar@/_5pc/[rrd]^g & \Hom(U, V) \ar@{>>}[r] & \Hom (U,Q) \ar[d]^\simeq \\
& \Hom(V^*,U^*) \ar@{>>}[r] & \Hom(Q^*, U^*)
}
$$
We know the map $f$ to be surjective by Lemma \ref{lem-tec3}, hence we have that for every $Q^* \subset V$ with $\dim Q = k+1$ the map $g$ is also surjective. This means that the map $\varphi '$ makes $W$ satisfy also the reciprocal of the property $P_k$, because every morphism induced by $\varphi '$ given by the restriction of $V^*$ to a $(k+1)$-dimensional subspace is surjective.
\end{proof}
The lemma we just proved applies to our particular case in the following way.
\begin{Remark}
By the description of the property $P_k$, we have that $P_k$ implies $P_i$, for each $i\leq k$. Moreover, Lemma \ref{lem-dual} tells us that having a Steiner bundle $F$ on $\GG(k,n)$ is equivalent to have a Steiner bundle $\tilde{F}$ on $\GG(k,\PP(S))$. Therefore considering a Steiner bundle actually means considering a family of $2(k+1)$ Steiner bundles.
\end{Remark}
We now prove a result that will allow us to extrapolate the trivial summands from Steiner bundles and define the concept of \emph{reduced Steiner bundle} on $\GG(k,n)$, generalizing the definition given by one of the authors in \cite{arr}.
\begin{Lemma}\label{lem-rid}
With the usual notation, the followings are equivalent:
\begin{description}
\item[(i)] a linear subspace $K \subset T^*$ contained in the kernel of $\varphi$,
\item[(ii)] an epimorphism $F\longrightarrow K^* \otimes \OO_\GG$,
\item[(iii)] a splitting $F=F' \oplus \left(K^* \otimes \OO_\GG\right)$.
\end{description}
\end{Lemma}
\begin{proof} { \ } \vspace{4mm}\\
(ii) $\Leftrightarrow$ (iii)\\
Let us observe that $F$ is generated by its global sections (that is because we have a surjective map from a vector space tensor the canonical bundle to it). Let us consider the following diagram
$$
\xymatrix{
0 \ar[r] & S \otimes \UU \ar[r] & T \otimes \OO_\GG \ar[r] \ar[dr] & F \ar[r] \ar[d]^{f} & 0 \\
&&&K^* \otimes \OO_\GG  \\
}
$$
The fact that $F$ is generated by its global sections tells us that a basis will be sent to a basis, so, let us take a basis of $H^0(F)$ and suppose that $f$ is surjective. What we need to do is to take a basis of $K^*\otimes \OO_\GG$ and take its preimage in $F$. We can complete the basis and with the elements added we can generate a vector bundle $F'$ that gives us the splitting.\\
If we already know we have a splitting, of course the map $f$ is surjective.\vspace{5mm}\\
(i) $\Rightarrow$ (ii) \vspace{3mm}\\
If (i) is true, then, because of the fact that $K \subset \ker \varphi$, we have a morphism of the form $\bar{\varphi} : T^* / K \longrightarrow S^* \otimes H^0(\UU^\lor)$. We know well by now that such morphism is associated to a Steiner bundle $F'$. This gives us the following commutative diagram, that we can construct starting from the two rows
\begin{equation}\label{diag-lemred}
\xymatrix{
& & 0 \ar[d] & 0 \ar[d] \\
0 \ar[r] & S \otimes \UU \ar[r] \ar@{=}[d]  & \left(T^* / K\right)^* \otimes \OO_\GG \ar[r] \ar@{^{(}->}[d]^\alpha & F' \ar[r] \ar[d]^\beta & 0 \\
0 \ar[r] & S \otimes \UU \ar[r]  & T \otimes \OO_\GG \ar[d] \ar[r] & F \ar[r] \ar[d]^\gamma & 0 \\
& & K^* \otimes \OO_\GG \ar@{=}[r] \ar[d]  & K^* \otimes \OO_\GG \ar[d] \\
& & 0 & 0
}
\end{equation}
The morphism $\alpha$ is injective by hypothesis and it gives us that also $\beta$ is injective. Because of the commutativity of the diagram, the map $\gamma$ must be surjective. We thus have (ii). \vspace{3mm}\\
(ii) $\Rightarrow$ (i)\\
By hypothesis we have an epimorphism $F\longrightarrow K^*\otimes \OO_\GG$ and taking the resolution of $F$ we can construct a further surjective morphism $g$ in the following way:
$$
\xymatrix{
\cdots \ar[r] &T \otimes \OO_\GG \ar@{>>}[r] \ar[dr]^g & F \ar[r] \ar@{>>}[d] &0 \\
&& K^* \otimes \OO_\GG
}
$$
Having such surjective morphism $g$ tells us that $K \subset T^*$ must be a vector subspace. Starting from the dual of the map $g$ we are able to construct another commutative diagram, the dual one of Diagram (\ref{diag-lemred}), this time starting from the two columns.\\ 
We get that the bundle morphism $\varphi : T^* \otimes \OO_\GG \longrightarrow S^* \otimes \UU^\lor$ has a factorization through  $(T^* / K ) \otimes \OO_\GG$. This means that $K$ is contained in the kernel of $\varphi$ and we get (i).
\end{proof}
\begin{Remark}\label{rmk-reduced}
The bundle $F'$, by Lemma \ref{lem-def}, is the Steiner bundle associated to the map $T^*/ K \longrightarrow S^*\otimes H^0(\UU^\lor)$. If we consider $T_0^* = \im \varphi$, we obtain an inclusion $T_0^* \hookrightarrow S^* \otimes H^0(\UU^\lor)$ associated to a Steiner bundle $F_0$. We get $H^0(F_0^\lor) =0$ and $F = F_0 \oplus \left(T / T_0 \otimes \OO_\GG\right)$. In particular, if we fix a Steiner bundle $F$, $H^0(F^\lor) = 0$ if and only if the linear map $\varphi$ is injective.
\end{Remark}
All the Steiner bundles satisfying the property stated in the previous remark form a particular subset.
\begin{Definition}\index{Steiner bundle}\index{Steiner bundle ! reduced}
Using the usual notation, a Steiner bundle $F$ on $\GG(k,n)$ is said to be reduced if and only if $H^0(F^\lor) = 0$. In general, we will denote by $F_0$ the reduced summand of a Steiner bundle $F$.
\end{Definition}
Having defined the concept of reduced Steiner bundle, we can prove a classification theorem for $(s,t)$-Steiner bundles with $s \leq k+1$.
We have the following result.
\begin{Theorem}\label{teo-classtriv}
Let $F$ be an $(s,t)$-Steiner bundle on $\GG(k,n)$ with $s \leq k+1$. Then, if $F$ is reduced, it will be of the type $F = S \otimes \QQ$; if it is not it will be of the type $F = (S \otimes \QQ) \oplus \OO_\GG^p$ for some $p>0$.
\end{Theorem}
\begin{proof}
From Remark \ref{rmk-reduced}, it is enough to prove the result for reduced Steiner bundles; and, in this case, we need to show that the injective map $T^*\to S^*\otimes H^0(\UU^\lor)$ is also surjective. Consider the following commutative diagram, constructed starting from the dual of the sequence that defines the bundle,
$$
\xymatrix{
& 0 \ar[d] & 0 \ar[d] \\
0 \ar[r] & F^\lor \ar[r] \ar[d] & T^* \otimes \OO_{\GG(k,n)} \ar[d] \ar[r] & S^* \otimes \UU^\lor \ar[r] \ar[d]^\simeq & 0 \\
0 \ar[r] & S^* \otimes \QQ^\lor \ar[d] \ar[r] & S^* \otimes H^0(\UU^\lor) \otimes \OO_{\GG(k,n)} \ar[r] \ar[d] & S^* \otimes \UU^\lor \ar[r] & 0 \\
& \frac{S^* \otimes H^0(\UU^\lor)}{T^*} \otimes \OO_{\GG(k,n)} \ar[r]^\simeq \ar[d] & \frac{S^* \otimes H^0(\UU^\lor)}{T^*} \otimes \OO_{\GG(k,n)} \ar[d] \\
& 0 & 0}
$$
Observe that, if $T^* \neq S^*\otimes H^0(\UU^\lor)$, the given diagram allows us to have an injective morphism of type
$$
\OO_\GG \longrightarrow (\QQ)^s \simeq S \otimes \QQ.
$$
If $s \leq k+1$ this is impossible; indeed,  a global section of $\QQ$ vanishes in all $\Lambda \in \GG(k,n)$ passing through a fixed point of $\PP^n$. Considering $s$ sections of $\QQ$ means considering all $\Lambda$'s passing through $s$ independent points of $\PP^n$. If  $s \leq k+1$ then we have at least one element with such property. Hence $T^*= S^*\otimes H^0(U^\lor)$, which completes the proof.
\end{proof}
\begin{Remark}
Observe that, from the previous diagram, the Steiner bundle $F$ can be also defined considering the strongly exceptional pair $(\OO_\GG, \QQ)$. This proves that all the theory presented in this work does not depend on the choice of the strongly exceptional pair, as we pointed out at the beginning of this section.
\end{Remark}
\section{Schwarzenberger bundles on $\GG(k,n)$}\label{sec-Schw}
In this section we propose the definition and some properties of Schwarzenberger bundles on $\GG(k,n)$, which represents the generalization of the definition given in \cite{arr} of Schwarzenberger bundle on the projective space.\\
Let us consider two globally generated vector bundles $L,M$ over a projective variety $X$, with $h^0(M) = n+1$ and with the identification $\PP^n = \PP(H^0(M)^*)$. The Schwarzenberger bundle on $\GG(k,n)$ associated to the triple $(X,L,M)$ will be the bundle defined by the composition 
$$
H^0(L) \otimes \UU \longrightarrow H^0(L) \otimes H^0(M) \otimes \OO_\GG \longrightarrow H^0(L\otimes M) \otimes \OO_\GG
$$
for which we need to require the injectivity in each fiber, in order to have a resolution.\\
This is equivalent to fix $k+1$ independent global sections $\{\sigma_1,\ldots,\sigma_{k+1}\}$ in $H^0(M)$ in correspondence to the point $\Gamma = [<\sigma_1,\ldots,\sigma_{k+1}>] \in \GG(k,n)$ and require the injectivity of the following composition, given by the multiplication with the global section subspace we fixed
$$
H^0(L) \otimes <\sigma_1,\ldots,\sigma_{k+1}> \longrightarrow H^0(L) \otimes H^0(M) \otimes \OO_\GG \longrightarrow H^0(L\otimes M) \otimes \OO_\GG
$$
\begin{Definition}\label{def-Schw}
The bundle $F= F(X,L,M)$ defined by the resolution
$$
0 \longrightarrow H^0(L) \otimes \UU \longrightarrow H^0(L\otimes M) \otimes \OO_\GG \longrightarrow F \longrightarrow 0
$$
is called a Schwarzenberger bundle on $\GG(k,n)$.
\end{Definition}
The map $\varphi$ of Lemma \ref{lem-def} will be the dual of the multiplication map $H^0(L) \otimes H^0(M) \longrightarrow H^0(L\otimes M)$. In particular $F$ is reduced if and only if the multiplication map is surjective.\\
Notice that it is not easy to find examples of Schwarzenberger bundles on $\GG(k,n)$, nevertheless, the following lemma gives us a family which ensures us that the definition is correct.
\begin{Lemma}
Consider a triple $(X,L,M)$ such that $L$ and $M$ are globally generated vector bundles of rank respectively 1 and $k+1$ over a projective variety X, with $\dim X \geq k+1$, $c_{k+1}(M) \neq 0$ and $h^0(M) = k+2$. Then, the triple defines a Schwarzenberger bundle on $\GG(k,k+1)$.
\end{Lemma}
\begin{proof}
\hspace{3mm} Notice that for each $\Delta \subset H^0(M)$ of dimension $k+1$, the map $\Delta \otimes \OO_X \longrightarrow M$ is injective seen as a morphism of sheaves. Suppose that it is not and take $s_1\ldots,s_{k+1}$ generators of the subspace, then we will have that for each $x \in X$ the vectors given by $s_1(x),\ldots,s_{k+1}(x)$ are linearly dependent. Consider a further section $s_{k+2} \in H^0(M)$ which completes a basis for the space of the global sections, we know that its zero locus is non empty by the assumptions on the dimension of $X$ and on $c_{k+1}(M)$. Taking a point of this locus, we obtain that the evaluation map $H^0(M) \otimes \OO_X \longrightarrow M$ cannot be surjective, which leads to contradiction because we have supposed that $M$ is globally generated.
\end{proof}
\begin{Remark}
Notice that the hypothesis $\dim X \geq k+1$ is necessary. Indeed, if we consider $X= \PP^1$ and $M = \OO_{\PP^1}(1) \bigoplus_{i=1}^k \OO_{\PP^1}$, and taking $\Delta$ generated by the two independent global sections of $\OO_{\PP^1}(1)$ and $k-1$ independent sections given by the other summands, then the morphism of sheaves $\Delta \otimes \OO_{\PP^1} \longrightarrow \OO_{\PP^1}(1)  \bigoplus_{i=1}^k \OO_{\PP^1}$ is not injective.
\end{Remark}
A particular case of the considered family is given by taking the triple $(X,L,M) = (\PP^k,\OO_{\PP^k}(1),T_{\PP^k}(-1))$. This example will be of extreme importance in the classification, as we will see in Section \ref{sec-class}.\\
\begin{Remark}
Let us give the geometrical interpretation of a Schwarzenberger bundle on $\GG(k,n)$, trying to recover the one given for the Steiner bundles, seen in Remark \ref{intStein}. Consider the Schwarzenberger bundle $F = F(X,L,M)$. Moreover, let us consider a point $\Gamma \in \GG(k,n)$ which is associated to a $(k+1)$-dimensional subspace $\Delta$ of independent global sections of $M$, and hence to a subspace $A \subset \PP(H^0(M))$ of codimension $k+1$. Let us suppose that $A$ is defined by the set of equations $\{y_1 = \ldots = y_{k+1} = 0\}$, with an appropriate choice of the coordinates $(y_1:\ldots:y_{n+1})$ for the projective space $\PP(H^0(M))$. Moreover, to each $f \in H^0(L)$ we can obtain an hyperplane $H_f \subset \PP(H^0(L))$ and we can suppose it is defined by the equation $\{x_1 = 0\}$, after choosing proper coordinates $(x_1:\ldots:x_s)$ for $\PP(H^0(L))$. This two subspaces define a further linear subspace of codimension $k+1$ in $\PP(H^0(L) \otimes H^0(M))$ and hence a subspace $\tilde{A}_f$ of codimension $k+1$ in $\PP(H^0(L\otimes M))$. Indeed, $\tilde{A}_f$ is defined by the vanishing of the equations $\{x_1 y_j\}_{j=1}^{k+1}$. Recall that by hypothesis we have that $f \cdot \Delta \longrightarrow H^0(L\otimes M)$ is injective, which ensures that all products $x_1 y_j$ are independent seen as linear forms of $\PP(H^0(L\otimes M))$.\\
We are able to define $\tilde{A}_f$ considering the Segre map
$$
\nu : \PP(H^0(L)) \times \PP(H^0(M)) \longrightarrow \PP(H^0(L) \otimes H^0(M)),
$$
indeed we have that
$$
\tilde{A}_{f} = \left\langle \nu\left(H_f \times \PP(H^0(M))\right) \cup \nu\left( \PP(H^0(L)) \times A\right)\right\rangle \cap \PP(H^0(L\otimes M)),
$$
where we look at $\PP(H^0(L\otimes M))$ as a linear subspace of $\PP(H^0(L) \otimes H^0(M))$.\\
Recalling the geometric interpretation of a Steiner bundle, we get that the projectivization of the fiber of $F$ over the point $\Gamma$ is given as the intersection of all $\tilde{A}_f$ for each $f \in H^0(L)$. We thus obtain that 
$$
\PP(F_\Gamma) = \left\langle \nu \left(\PP(H^0(L)) \times A \right) \right\rangle \cap \PP(H^0(L \otimes M)).
$$
\end{Remark}

\section{Jumping pairs of a Steiner bundle}
In this section we will first introduce the concept of jumping pair for a Steiner bundle on $\GG(k,n)$. This definition generalizes both the one proposed by one of the authors in \cite{arr} and the concept of \emph{unstable hyperplane} presented and used in \cite{anc-ott} and \cite{val1}. We will then study the set of the jumping pair, describing it as a projective variety and estimating its dimension.\\
Let us consider a Schwarzenberger bundle given by the triple $(X,L,M)$, imposing $\rk L = 1$ and $\rk M = k+1$, as for the case of the projective space, these will be the more significant examples. 

Recall that is the linear map $\varphi$ defining the bundle $F$ is the dual of the multiplication map:
$$
H^0(L\otimes M)^* \longrightarrow H^0(L)^* \otimes H^0(M)^*,
$$
so that for any point $x\in X$ we have that the image of $H^0(L_x\otimes M_x)^*$ is $H^0(L_x)^* \otimes H^0(M_x)^*$. This motivates the following: \begin{Definition}\label{def-jump}
Let $F$ be a Steiner bundle over $\GG(k,n)$. A pair $(a,\Gamma) \in G (1,S^*) \times G(k+1,H^0(\UU^\lor))$ is called a jumping pair if, considering the linear map $T^* \stackrel{\varphi}{\rightarrow} S^* \otimes H^0(\UU^\lor)$, the tensor product $a \otimes \Gamma$ belongs to $\im \varphi$.
\end{Definition}
Let us fix some notation: we will denote by $\tilde{J}(F)$ the set of jumping pairs of $F$. Moreover, we have two natural projections $\tilde{J}(F)\stackrel{\pi_1}{\rightarrow} G(1,S^*)$ and $\tilde{J}(F)\stackrel{\pi_2}{\rightarrow} G(k+1,H^0(\UU^\lor))$. We will denote by $\Sigma(F)$ the image of the first projection and by $J(F)$ the image of the second one. This last set is called the set of the \emph{jumping spaces} of $F$.\\
\begin{Lemma}\label{lem-jumprid}
Let $F$ be a Steiner bundle over $\GG(k,n)$ and let $F_0$ be its reduced summand, then $\tilde{J}(F) = \tilde{J}(F_0)$.
\end{Lemma}
\begin{proof}
Recall that we have a splitting of the bundle that gives $F^* = F_0^* \oplus ((T^* / T_0^*) \otimes \OO_\GG)$, where $T_0^* = \im \varphi$, and hence $H^0(F^*) = H^0(F_0^*) \oplus (T^* / T_0^*)$, with $H^0(F_0^*)=0$. Obviously if $(a,\Gamma)$ is a jumping pair for $F_0$, then it is also a jumping pair for $F$. Using the definition we prove the viceversa. Indeed, we know that a jumping pair is defined as $0 \neq (a \otimes \Gamma) \subset S^* \otimes H^0(\UU^\lor)$ such that it exists $\Lambda \subset T^*$ with $\varphi(\Lambda) = a \otimes \Gamma$. If $\Lambda \cap (T^* / T_0^*) \neq \emptyset$ then $\dim a\otimes \Gamma < k+1$ because of the exact sequence defining $\varphi$. We thus have that $\Lambda \subset T_0^*$ and $a\otimes \Gamma$ is a jumping pair also for $F_0$.
\end{proof}
The previous lemma tells us that the jumping locus of a Steiner bundle coincides with the one of its reduced summand.\\
Our goal is to find a triple in order to define a Steiner bundle as a Schwarzenberger. The following result will show us how to construct such triple.
\begin{Lemma}\label{lem-varproj}
Let $F$ be a Steiner bundle on $\GG(k,n)$ and $T_0^* \subset S^* \otimes H^0(\UU^\lor)$ be the image of $\varphi$ (or equivalently the vector space associated with the reduced summand $F_0$ of $F$).\\
Consider the Segre generalized embedding
$$
\begin{array}{rccccc}
\nu: & G (1,S^*)& \times & G(k+1,H^0(\UU^\lor)) & \longrightarrow & G(k+1,S^* \otimes H^0(\UU^\lor))\\
& a & , & \Gamma & \mapsto & a \otimes \Gamma
\end{array}
$$
Then
\begin{description}
\item[(i)] we have that
$$
\tilde{J}(F) = \im \nu \cap G(k+1, T_0^*),
$$
\item[(ii)] Let $\A, \BB, \HH$ be the universal bundles of ranks respectively $1,k+1$ and $k+1$ over $G(1,S^*), G(k+1,H^0(\UU^\lor))$ and $G(k+1,T_0^*)$.\\
Consider the projections
$$
\xymatrix{
& \tilde{J}(F) \ar[dl]_{\pi_1} \ar[dr]^{\pi_2} \\
G(1,S^*) & & G(k+1,H^0(\UU^\lor))
}
$$
and that $\tilde{J}(F) \subset G(k+1,T_0^*)$. Assume that the natural maps
$$
\begin{array}{rlcl}
\alpha : & H^0(G(1,S^*),\A) & \longrightarrow & H^0(\tilde{J}(F), \pi_1^*\A)\\
\beta : & H^0(G(k+1, H^0(\UU^\lor)),\BB) & \longrightarrow & H^0(\tilde{J}(F), \pi_2^*\BB)\\
\gamma : & H^0(G(k+1,T_0^*),\HH) & \longrightarrow & H^0(\tilde{J}(F), \HH_{|\tilde{J}(F)})
\end{array}
$$
are all isomorphisms. Then the Steiner bundle $F_0$, reduced summand of $F$, is a Schwarzenberger bundle given by the triple 
$$
(\tilde{J}(F),\pi_1^* \A, \pi_2^* \BB).
$$
\end{description}
\end{Lemma}
\begin{proof}
Notice that part (i) is just the geometrical interpretation of the definition of jumping pair we have given.\\
To prove part (ii), consider the commutative diagram
$$
\xymatrix{
S \otimes H^0(\UU^\lor)^* \ar[r] \ar[d] & T_0 \ar[d]\\
H^0(\tilde{J}(F), \pi_1^* \A) \otimes H^0(\tilde{J}(F),\pi_2^* \BB) \ar[r] &H^0(\tilde{J}(F),\pi_1^* \A \otimes \pi_2^* \BB)
}
$$
The top map is the dual of the inclusion $T_0^* \hookrightarrow S^* \otimes H^0(\UU^\lor)$ that defines the reduced summand of a Steiner bundle; such map can be identified with the following
$$
H^0(G(1,S^*),\A) \otimes H^0(G(k+1,H^0(\UU^\lor)),\BB) \longrightarrow H^0(G(k+1,T_0^*),\HH)
$$
defined by the composition of the multiplication map and the restriction of the global sections, due to the fact that $T_0^* \subset S^* \otimes H^0(\UU^\lor)$,
$$
H^0(G(k+1,S^* \otimes H^0(\UU^\lor)),\tilde{\HH}) \longrightarrow H^0(G(k+1,T_0^*),\HH),
$$
where of course $\tilde{\HH}$ denotes the universal bundle over $G(k+1,S^* \otimes H^0(\UU^\lor))$.
The vertical maps, due to the last identification, are $\alpha \otimes \beta$ and $\gamma$, which are isomorphisms by hypothesis.\\
The bottom map is the multiplication map whose dual defines a Schwarzenberger bundle.\\
Because of the isomorphisms in the diagram, we can state that $F_0$ is Schwarzenberger defined by the triple $(\tilde{J}(F),\pi_1^* \A, \pi_2^* \BB)$ and this concludes the proof.
\end{proof}
Let us fix some notation: we will consider a jumping pair both as vectorial and projective. This means that we will study, in order to see the locus as a projective variety, the pairs $(\PP(a^*),\PP(\Gamma^*))$ whose tensor product belongs to $\PP(S \otimes H^0(\UU^\lor)^*)$. In order to do so consider the generalized Segre map
$$
\begin{array}{rccccc}
\nu: & \PP(S)& \times & \GG(k,\PP(H^0(\UU^\lor))^*) & \longrightarrow & \GG(k,\PP(S \otimes H^0(\UU^\lor)^*)) := \bar{\GG}\\
& \PP(a^*) & , & \PP(\Gamma^*) & \mapsto & \PP(a^* \otimes \Gamma^*)
\end{array}
$$
and by Lemma \ref{lem-varproj} we have that $\tilde{J}(F) = \im \nu \cap \GG(k,\PP(T_0))$. Notice that we take $\PP(T_0)$ because of what we proved in Lemma \ref{lem-jumprid}.\\
To limit the dimension of $\tilde{J}(F)$ from below we need then to compute the expected dimension of the intersection, i.e. when we have a complete intersection, which is
\begin{equation}\label{lowdim}
\dim \tilde{J}(F) \geq (k+1)(t-k-sn-s+n) + s -1.
\end{equation}
Observe that such inequality becomes an equality for a general Steiner bundles, hence the general Steiner bundle cannot be described as a Schwarzenberger.\\
In order to find a limit from above, we will fix a jumping pair $\Lambda \in \tilde{J}(F)$ and we will consider the tangent space $T_\Lambda \tilde{J}(F)$ of the jumping variety at the fixed point. First of all we need to find a proper description of such tangent space, then, through more general results of linear algebra, we will find the requested bound.
\begin{Theorem} \label{teo-tangjump}
Let $F$ be a Steiner bundle over $\GG(k,n)$ and let $\Lambda \in \tilde{J}(F)$ be one of its jumping pairs; then 
\begin{equation}\label{eq-tang}
T_{\Lambda} \tilde{J}(F) = \left\{ \psi \in \Hom \left(\Lambda,\frac{T_0^*}{\Lambda}\right)|
{\begin{array}{cc}
(\psi(\varphi_i))(\ker \varphi_i) \subset <v_1>  \\
(\psi(\varphi_i))(u_i) \equiv (\psi(\varphi_j))(u_j) \mod v_1
\end{array} }
\right\},
\end{equation}
where $v_1 \in S^*$ and the sets $\{\varphi_i\}_{i=1}^{k+1}$ and $\{u_j\}_{j=1}^{n+1}$ are basis respectively of $\Lambda$ and $H^0(\UU^\lor)^*$, properly chosen as we will see in the proof of the next theorem. 
\end{Theorem}
This last result is a consequence of the following theorem which describes the tangent space of the generalized Segre variety at the fixed jumping pair $\Lambda \in \tilde{J}(F)$ (we will denote by $\Seg$ the image of $\nu$).
\begin{Theorem}\label{teo-tangseg}
Let $\Seg$ denote the image of the generalized projective Segre embedding and let $\tilde{J}(F)\subset \Seg$ be the set of the jumping pairs of a Steiner vector bundle $F$ over $\GG(k,n)$. Fixing a point $\Lambda = s_0\otimes \Gamma \in \tilde{J}(F)$, we have that
$$
T_\Lambda \Seg :=  \left\{ \psi \in \Hom \left(\Lambda, \frac{\Hom(H^0(\UU^\lor)^*,S^*)}{\Lambda} \right) |\:\forall\: \varphi \in \Lambda, (\psi(\varphi))(\ker \varphi) \subset <s_0>   \right\}
$$
where $T_\Lambda \Seg$ denotes the tangent of the Segre image at the point $\Lambda$.
\end{Theorem}
\begin{proof}
First of all, observe that the description of the tangent space is equivalent to the following one
$$
T_{\Lambda} \Seg = \left\{ \psi \in \Hom \left(\Lambda,\frac{\Hom(H^0(\UU^\lor)^*,S^*)}{\Lambda}\right)|
{\begin{array}{cc}
(\psi(\varphi_i))(\ker \varphi_i) \subset <v_1>  \\
(\psi(\varphi_i))(u_i) \equiv (\psi(\varphi_j))(u_j) \mod v_1
\end{array} }
\right\}
$$
after choosing properly basis for $\Lambda$, $H^0(\UU^\lor)^*$ and $S^*$ (the basis will be described later in the proof).\\
Consider a jumping point $\Lambda$ that we can suppose to be the origin. Consider also the equations which locally define the Segre generalized variety and look, in such defining equations, for their linear part.  In this way, we obtain $(s-1)(nk+n+k)$ independent conditions (for more details, see \cite{mar-tesi}).\\
We can observe that the number of conditions we found is exactly the one needed to define the tangent space.\\  
The notation used is the following: a jumping pair $\Lambda$ is equal to a tensor product $s_0 \otimes \Gamma \subset S^* \otimes H^0(\UU^\lor) = \Hom(H^0(\UU^\lor)^*,S^*)$. We can suppose that, once we fixed $\Lambda$, for every choice of a basis $v_1,\ldots,v_s$ of $S^*$ we can take $v_1 = s_0$. Such vector represents the image of all the rank 1 elements in $\Lambda$, seen as morphism from $H^0(\UU^\lor)^*$ to $S^*$.\\
To prove that the vector space defined by all the linear forms extracted is actually the tangent space, we will compute its dimension. It will be sufficient to compute the dimension of the following vector space
$$
\tilde{T} := \left\{ \psi \in \Hom \left(\Lambda, \Hom(H^0(\UU^\lor)^*,S^*) \right) |\:\forall\: \varphi \in \Lambda, (\psi(\varphi))(\ker \varphi) \subset <v_1> \right\}
$$
Let us consider the proper subspace of $\tilde{T}$ given by
$$
K = \left\{\psi \in \Hom (\Lambda, \Hom (H^0(\UU^\lor)^*,<v_1>))\right\}
$$
with of course $\dim K = (k+1)(n+1)$. We obtain the following commutative diagram
$$
\xymatrix{
0 \ar[r] & K \ar@{^{(}->}[r] & \tilde{T} \ar[dr] \ar[r] & P \ar[r] \ar@{^{(}->}[d]& 0 \\
&&& \Hom\left( \Lambda, \Hom \left( H^0(\UU^\lor)^*, \frac{S^*}{<v_1>}\right)\right)
}
$$
where $P \subset \Hom\left( \Lambda, \Hom \left( H^0(\UU^\lor)^*, \frac{S^*}{<v_1>}\right)\right)$ represents the subset of morphisms whose images are all rank 1 elements of the set of morphisms $\Hom \left( H^0(\UU^\lor)^*, \frac{S^*}{<v_1>}\right)$. Observe that $P$ is also a vector space, being the quotient of two vector spaces. Our next goal is to compute the dimension of $P$.\\
Consider a family of morphisms $g_i : S^* \longrightarrow S^*$ such that $g_i(v_1)=v_i$, for $i=2,\ldots,s$ and define $s-1$ morphisms $\psi_i(\varphi) := g_i \circ \varphi$. We get then the commutative diagrams, one for each $i$,
$$
\xymatrix{
H^0(\UU^\lor) ^* \ar[r]^{\varphi} \ar[d]_{\simeq} &  <v_1> \subset S^* \ar@<-3ex>[d]^{g_i} \\
H^0(\UU^\lor) ^* \ar[r]^{\psi_i(\varphi)}\ar@<-1ex>[dr] &  <v_i> \subset S^* \ar@<-3ex>[d]^{\pi} \\
& <\bar{v}_i> \subset \frac{S^*}{<v_1>}
}
$$
where $\pi$ simply represents the quotient of $S^*$ by the subspace $<v_1>$. Observe that all the images of the elements $g_i(v_1)$ through the projection $\pi$ are independent by definition. In this way we have constructed the maps $\psi_2,\ldots,\psi_{s}$ and we want to prove they form a basis for $P$.\\
Let us fix a basis $u_1,\ldots,u_{n+1}$ for $H^0(\UU^\lor)^*$ and let us build a basis $\varphi_1,\ldots,\varphi_{k+1}$ of $\Lambda$ with the property that $\varphi_i(u_i) = v_1$ and $\varphi_i(u_j) = 0$, for every $i$ from $1$ to $k+1$ and every $j$ from 1 to $n+1$, with $j\neq i$.\\
Let us take the generic element $\psi \in P$ and check how it behaves when applied to the elements of the chosen basis of $\Lambda$. Consider two such elements $\varphi_i$ and $\varphi_j$, with $i\neq j$. By hypothesis we know that $(\psi(\varphi_i))(H^0(\UU^\lor)^*) \subset <\tilde{v}_i>$ with $\tilde{v}_i \in \frac{V}{<v_1>}$ and also $(\psi(\varphi_j))(H^0(\UU^\lor)^*) \subset <\tilde{v}_j>$ with $\tilde{v}_j \in \frac{V}{<v_1>}$ . Furthermore we know that, considering the element given by the sum $\varphi_i + \varphi_j$, we have $(\psi(\varphi_i+\varphi_j))(H^0(\UU^\lor)^*) \subset <\tilde{v}>$ with $\tilde{v} \in \frac{V}{<v_1>}$, because the morphism $\psi(\varphi_i + \varphi_j)$ must have rank one and therefore exists a $\tilde{v}$ that gives us the last inclusion. We repeat this process for every pair of independent elements in the basis of $\Lambda$, so that we can state that we must get the same vector $\tilde{v} \in \frac{V}{<v_1>}$ for every $\varphi \in \Lambda$, i.e. $((\psi(\varphi))(H^0(\UU^\lor)^*) \subset <\tilde{v}>$.\\
We have that $\tilde{v} = \lambda_2 \tilde{v}_2 + \lambda_3 \tilde{v_3} + \ldots \lambda_s \tilde{v}_s$ where $\lambda_i \in \KK$ and $\tilde{v}_i$, for $i=2,\ldots,s$, are the vectors found defining the maps $\psi_i$ and moreover the previous equivalence is true for every $u \in H^0(\UU^\lor)^*$ and for every $\phi \in \Lambda$, which guarantees us that $\psi = \lambda_2 \psi_2 + \ldots +\lambda_s \psi_s$ for every $\psi \in P$.\\
We have thus proved that we have a basis for $P$, whose morphisms $\psi_p$ have the following behavior: we know that $(\psi_p(\varphi_i))(\ker \varphi_i) \subset <v_1>$  for every $i = 1,\ldots,k+1$ and to every morphism we can associate a vector $v_p$ with $\pi(v_p) = \tilde{v}_p \in \frac{S^*}{<v_1>}$ which is not zero, such that $(\psi_p(\varphi_i))(u_i) \subset <v_p>$. Considering the sum of two morphisms and its kernel $\ker(\varphi_i + \varphi_j) = <u_i - u_j ,\{u_{\bar{k}}\}_{\bar{k}\neq i,j}>$ we also get $(\psi_p(\varphi_i + \varphi_j))(\ker (\varphi_i + \varphi_j)) \subset <v_1>$; so we can extend such properties to every element of $\Lambda$.\\
To conclude we have proved that $\dim P = s-1$, hence $\dim \tilde{T} = (k+1)(n+1)+s-1.$\\
We have, by definition, that $\dim T_\Lambda \Seg = \dim \tilde{T} - \dim \End(\Lambda) = (k+1)(n-k) + s-1 = \dim (\PP(S) \times \GG(k,n))$; so it is actually the tangent space we were looking for and this completes the proof.
\end{proof}
\begin{Remark}
The description of the tangent space at a jumping point
$$
T_\Lambda \tilde{J}(F) :=  \left\{ \psi \in \Hom \left(\Lambda, \frac{T^*}{\Lambda} \right) |\:\forall\: \varphi \in \Lambda, (\psi(\varphi))(\ker \varphi) \subset <s_0>   \right\}
$$
implies that, for each $\psi \in T_\Lambda \tilde{J}(F)$, there exists a unique 2-dimensional subspace $A \subset S^*$, with $\langle s_0 \rangle \subset A$, such that $\im \psi(\varphi) \subset A$.
\end{Remark}
\subsection{The technical lemmas}\label{sub-tech}
We will prove now the results of linear algebra that will allow us to find the bound from above we were looking for.\\
Let us fix some notation: let $U$ be a vector space of dimension $r>k$ and $V$ be a vector space of dimension $s$. Let $v_1,\ldots,v_s$ be a basis for $V$. Consider the vector space
$$
\Lambda := \left\{ \varphi \in \Hom(U,V) | \im \varphi \subset <v_1> \:\mbox{and} \: \dim \left( \bigcap_{\varphi \in \Lambda} \ker \varphi \right) = n-k \right\}.
$$
We are able to construct a basis $u_1,\ldots,u_r$ of $U$ and a basis $\varphi_1,\ldots,\varphi_{k+1}$ of $\Lambda$ in order to have
$$
\ker \varphi_i = <u_1,\ldots,u_{i-1},u_{i+1},\ldots,u_r> =: U'_i,
$$
i.e. its kernel is the span of all vectors except $u_i$.\\
Recall that we denoted by $W \subset \Hom(U,V)$ the vector subspace characterized by the following property, that we called the property $P_k$: for every $k+1$ independent vectors $\tilde{u}_1,\ldots,\tilde{u}_{k+1} \in U$ and for every $\tilde{v}_1,\ldots,\tilde{v}_{k+1} \in V$ there exists an element $f \in W$ such that $f(\tilde{u}_j) = \tilde{v}_j$, for every $j = 1,\ldots,k+1$.\\
The main result we need to prove will be accomplished by two steps, each one stated as a technical lemma. Let us start with the first one. 
\begin{Lemma}\label{lem-tec1}
Let $U,V,\Lambda$ and $W$ be defined as before with the basis previously constructed. Let us consider
$$
\tilde{\Gamma} := \left\{ \psi \in \Hom(\Lambda,W) \: | \: (\psi(\varphi_i))(\ker \varphi_i) \subset <v_1>, \: i=1,\ldots,k+1\right\},
$$
then $\dim \tilde{\Gamma} \leq (k+1)(t-(k+1)(s+r-k-3))$ where we denote by $t =\dim W$.
\end{Lemma}
\begin{proof}
Let us fix an element $\varphi_i$ of the basis of $\Lambda$ and check out how many independent morphisms, to choose as its image, we can find. We will associate such possible morphisms as the image of $\varphi_i$ and choosing the identically zero morphism as the image of the other morphisms of the basis, we can construct independent elements of $\Hom(\Lambda,W)$ which are not in $\tilde{\Gamma}$ and independent modulo $\tilde{\Gamma}$. Then we will fix, one at a time, each element of the basis of $\Lambda$ and we will repeat such procedure. This technique allows us to give an upper bound for the dimension of $\tilde{\Gamma}$.\\
Remember that at the moment we have fixed one element $\varphi_i$ and we are looking for all its possible images.
Let us take $\tilde{u}_1,\ldots,\tilde{u}_{k+1}$ independent elements of $\ker \varphi_i$ and let us build
$$
\tilde{\varphi}_{ij} 
\left\{
\begin{array}{rcl}
\tilde{u}_i & \mapsto & v_j \\
\tilde{u}_l & \mapsto & 0
\end{array}
\right.
\:\:\:\mbox{for}\: i,l=1,\ldots,k+1; \: j =2,\ldots, s \: \mbox{and} \: l \neq i.
$$
By counting them, we observe we have constructed $(k+1)(s-1)$ independent morphisms modulo $\tilde{\Gamma}$. Let us denote the family of morphisms we have found as
$$
\Psi_{ij} := \left\{
\begin{array}{l}
\varphi_i \mapsto \tilde{\varphi}_{ij}\\
\varphi_l \mapsto 0
\end{array}
\right. 
\:\:\:\mbox{for}\:\: i,l=1,\ldots,k+1, \: j = 2,\ldots, s, \: l\neq i.
$$
We would now like to construct new independent morphisms by using the rest of the elements in the kernel of $\varphi_i$ other that the $k+1$ independent chosen vectors as we have done until now. This is why we now want to add further vectors $\{\tilde{u}_{q}\}_{q=k+2}^{r-1}$ belonging to $\ker \varphi_i$ and independent from the ones we have fixed before.\\ 
Suppose that we have already fixed independent vectors $\tilde{u}_1,\ldots,\tilde{u}_{k+1},\tilde{u}_{k+2},\ldots,\tilde{u}_{q+k+1}$ in $\ker \varphi_i$ and already found $(k+1)(s-1) + (k+1)q$ independent morphisms modulo $\tilde{\Gamma}$, with $q \leq r-k-2$. By the hypothesis on $W$, we can control $k+1$ independent vectors among the $q+k+2$ we are considering. Therefore, we will have new morphisms, independent modulo $\Gamma$, if the following bundle map is not surjective
$$
\xymatrix{
\OO_\GG^{(k+1)(s-1)+(k+1)q+p} \ar[rr] \ar[dr] & & \frac{V}{<v_1>} \otimes \UU^\lor \ar[dl] \\
& G(k+1,q+k+2)
}
$$
By Porteous formula, if we consider the points for which the restriction of the morphism between vector bundles has no maximal rank, then their expected codimension will be $(k+1)q+p+1$ in an ambient space of dimension $\dim G(k+1,q+k+2) = (k+1)(q+1)$. As long as $p+1 \leq k+1$, we can add a new independent morphism; moreover we can repeat this process until we find enough independent vectors that span $\ker \varphi_i$. Let us denote by $\hat{\varphi}_{i,q,p}$ the independent morphism we obtain when we have already fixed $q$ elements in $\ker \varphi_i$ and we have already found $p-1$ new morphisms for this step. Notice that by construction, at each step there exists a vector $\bar{u}_{iqp}$, not belonging to the span $\langle \tilde{u}_1,\ldots, \tilde{u}_{k+1},\tilde{u}_{k+2},\ldots,\tilde{u}_{q-1}\rangle$, such that $\hat{\varphi}_{iqp}(\bar{u}_{iqp})$ is linearly independent, modulo $v_1$, with respect to the set $$\left\{\hat{\varphi}_{ihk}(\bar{u}_{iqp}), \tilde{\varphi}_{ij}(\bar{u}_{iqp}) \:| \: i=1,\ldots, k+1; h \leq q; k = 1,\ldots, k+1 \:\mbox{and}\: k < p \:\mbox{if}\: h = q \right\}.$$
Using the morphisms obtained we are able to construct a family
$$
\Delta_{iqs}:=
\left\{
\begin{array}{l}
\varphi_i \mapsto \hat{\varphi}_{iqp} \\
\varphi_l \mapsto 0
\end{array}
\right.
\:\mbox{for every}\;
\begin{array}{l}
i,l=1,\ldots,k+1\:\:\mbox{with} \:\: l \neq i,\\
 q=k+2,\ldots,r-1\: \mbox{and}\:p=1,\ldots, k+1.
 \end{array}
$$
Observe that we have found $(k+1)(s-1)+(k+1)(r-k-2)$ independent images for each fixed $\varphi_i$ element of the basis of $\Lambda$ and then we have decided to send the other elements of the basis to the zero morphism. Repeating such procedure for each fixed element $\varphi_i$, with $i$ from 1 to $k+1$, we have obtained
$$
(k+1)^2(r+s-k-3)
$$
independent morphisms modulo $\tilde{\Gamma}$. Indeed, consider their linear combination 
$$
\sum_{i=1}^{k+1} \sum_{j=2}^s \lambda_{ij} \Psi_{ij} + \sum_{i=1}^{k+1} \sum_{q=k+2}^{r-1} \sum_{p=1}^{k+1} \mu_{iqp} \Delta_{iqp} = f,
$$
with $f \in \tilde{\Gamma}$, and apply it to a fixed element $\varphi_i$ of the basis of $\Lambda$, in order to get
$$
 \sum_{j=2}^s \lambda_{ij} \tilde{\varphi}_{ij} + \sum_{q=k+2}^{r-1} \sum_{p=1}^{k+1} \mu_{iqp} \hat{\varphi}_{iqp} = f(\varphi_i).
$$
Notice that, by construction, applying such combination to the vectors $\bar{u}_{iqp}$, starting from the top values of the subindex $p$ and $q$, makes all coefficients of type $\mu_{iqp}$ vanish. That is because the image of the corresponding morphism was independent modulo $v_1$ to the images of the previous ones and $(f(\varphi))(\bar{u}_{iqp}) \in <v_1>$. We are left with the combination
$$
\sum_{j=2}^s \lambda_{ij} \tilde{\varphi}_{ij} = f(\varphi_i)
$$
which, applied to element $\tilde{u}_i$, gives us
$$
\sum_{j=2}^s \lambda_{ij} v_j = (f(\varphi_i))(\tilde{u}_i) \subset <v_1>,
$$
hence $\lambda_{ij} =0$. Fixing each element $\varphi_i$, we manage to vanish all the coefficients in the linear combination, hence we have independency modulo $\tilde{\Gamma}$.\\
We can thus state that
$$
\dim \tilde{\Gamma} \leq (k+1)(t-(k+1)(r+s-k-3)),
$$
which concludes the proof. 
\end{proof}
Observe that the bound we just proved only involves one of the two conditions listed in expression (\ref{eq-tang}) that we found in the definition the tangent space given in Theorem \ref{teo-tangjump}. That is why we will now demonstrate a second result that will also involve the second condition.
\begin{Lemma}\label{lem-tec2}
Let $U,V,\Lambda$ and $W$ defined as before and let $\tilde{\Gamma}$ be the vector subspace defined in Lemma \ref{lem-tec1}. Let us consider
$$
\Gamma = \left\{ \psi \in \tilde{\Gamma} | (\psi(\varphi_i))(u_i) \equiv (\psi(\varphi_j))(u_j) \mod v_1, \: \mbox{with}\: i \neq j \:\mbox{from}\: 1 \:\mbox{to}\: k+1\right\},
$$
then $\dim \Gamma \leq \dim \tilde{\Gamma} - k(k+1)$.
\end{Lemma}
\begin{proof}
We will prove this lemma with the same idea we used to prove the previous one, i.e. we will look for morphisms in $\tilde{\Gamma}$ that are independent modulo $\Gamma$. Consider a fixed element $\varphi_i$ of the basis of $\Lambda$ and the following commutative diagram
$$
\xymatrix{
0 \ar[r] & \ker \alpha \ar[r] & W \ar[d] \ar[r]^(.3){\alpha} & \Hom\left(\ker \varphi_i, \frac{V}{<v_1>}\right) \ar[d] \\
0 \ar[r] & \ker \beta \ar@{^{(}->}[u] \ar[r] & W \ar[r]^(.3){\beta} & \Hom\left(U, \frac{V}{<v_1>}\right) 
}
$$
Notice that the number of the morphisms we are looking for is equal to the number of independent morphisms of $\frac{\ker \alpha}{\ker \beta}$, so we must compute $$\dim (\ker \alpha) - \dim (\ker \beta) = \dim(\im \beta) - \dim (\im \alpha).$$
For our vector space $V$ of dimension $s$, let us consider a quotient $$Q = \frac{V}{<v_1,\bar{v}_2,\ldots,\bar{v}_{s-k-1}>}$$ in order to have $\dim Q = k+1$ and construct elements $\{\tilde{v}_{s-k},\ldots,\tilde{v}_s\}$ which are linearly independent modulo $<v_1,\bar{v}_2,\ldots,\bar{v}_{s-k-1}>$; then take a look at the composition
$$
\xymatrix{
W \ar[r]^(.3){\beta} \ar@/_3pc/[rr]^f & \Hom(U,\frac{V}{<v_1>}) \ar[r]^{\pi} & \Hom(U,Q).
}
$$
By Lemma \ref{lem-tec3} we know that $f$ is surjective, so if we consider the $k+1$ independent morphisms
$$
\left\{\delta_j\right\}_{j=1}^{k+1} \subset \Hom\left(U,\frac{V}{<v_1>}\right) \:\:\mbox{with}\:\: \delta_j\left(\frac{U}{\ker \varphi_i}\right) = \tilde{v}_{s-k-1+j},
$$
we have that the set $\{\pi(\delta_j)\}_{j=1}^{k+1}$ is also independent considered in $\Hom(U,Q)$ and each morphism of the set must have at least one preimage in $W$, independent among them modulo $\im \alpha$ by construction. Observe that we found at least $k+1$ independent morphisms for the fixed element $\varphi_i$. If we use the explained technique for each $\varphi_i$ fixed, with $i=2,\ldots,k+1$, while we always send the element $\varphi_1$ to the zero morphism, in order to guarantee the independency modulo $\Gamma$, we get $k(k+1)$ morphisms that belong to $\tilde{\Gamma}$ and that are independent modulo $\Gamma$. We can thus say that
$$
\dim \Gamma \leq \dim \tilde{\Gamma} - k(k+1)
$$
which concludes the proof.
\end{proof}
Having proved this technical results we are finally ready to give the requested upper bound for the dimension of the tangent space.
\begin{Theorem}\label{thm-tang}
Let $F$ be a reduced Steiner bundle on $\GG(k,n)$ defined by the injective map $S \otimes \UU \longrightarrow T \otimes \OO_\GG$ and let $\Lambda$ be a jumping pair for $F$, i.e. $\Lambda \in \tilde{J}(F)$, then the dimension of the tangent space of the jumping locus variety at the point $\Lambda$ is bounded by
$$
\dim T_\Lambda\tilde{J}(F) \leq (k+1)(t-(k+1)(s+n-k-1)-k),
$$
where as usual $s = \dim S$ and $t = \dim T$.\\
\end{Theorem}
\begin{proof}
Notice that the geometrical description of the tangent space given in the expression (\ref{eq-tang}) is a set whose conditions satisfy exactly the ones requested in the hypothesis of Lemma \ref{lem-tec2}. Applying the lemma, with the values set by the bundle, gives us the result.
\end{proof}

\section{The classification}\label{sec-class}
In this section we will give the complete classification of Steiner bundles whose jumping locus has maximal dimension and we will find a triple $(X,L,M)$ in order to describe them as Schwarzenberger.\\ We will classify the case $s=k+2$, which will be fundamental for us, because we will present a technique of induction that will allow us to classify Steiner bundles starting from such basic case.\\
In Theorem \ref{teo-classtriv} we have seen that every $(s,t)$-Steiner bundle with $s\leq k+1$ is essentially trivial. In this section we classify the first non trivial case, when $s=k+2$. To prove such result we will use the following observation.
\begin{Remark}\label{rem-dual}
Recall that Lemma \ref{lem-def} and Lemma \ref{lem-dual} implied that every Steiner bundle $F$ on $\GG(k,n)$ is equivalent to a Steiner bundle $\tilde{F}$ on $\GG(k,\PP(S))$.
\end{Remark}
\begin{Theorem}
Let $F$ be a reduced Steiner bundle over $\GG(k,n)$, with $\dim S = k+2$, then $F$ can be seen as the Schwarzenberger bundle given by the triple $(\PP^{k+1}, \OO_{\PP^{k+1}}(1), E^\lor (-1))$, where we identify $\PP(S)=\PP^{k+1}$ and $E$ is the vector bundle defined as the kernel of the surjective morphism $$H^0(\UU^\lor) \otimes \OO_{\PP(S)}(-1) \longrightarrow \frac{S^* \otimes H^0(\UU^\lor)}{T^*} \otimes \OO_{\PP(S)}.$$
\end{Theorem}
Observe because of the bounds obtained in inequality (\ref{lowdim}) and Theorem \ref{thm-tang}, $\tilde{J}(F)$ is a smooth complete intersection of dimension $\dim \tilde{J}(F) = (k+1)(t-(k+1)(n+1)-k).$
\begin{proof}
Consider the following commutative diagram, observing that $\OO_{\GG(k, \PP(S^*)} \simeq \OO_{\PP(S)}$,
$$
\xymatrix{
& 0 \ar[d] & 0 \ar[d] \\
0 \ar[r] & E \ar[r] \ar[d] & T^* \otimes \OO_{\PP(S)} \ar[r] \ar[d] & T_{\PP(S)} (-1) \otimes H^0(\UU^\lor) \ar[r] \ar@{=}[d] & 0 \\
0 \ar[r] & H^0(\UU^\lor) \otimes \OO_{\PP(S)}(-1) \ar[r] \ar[d] & S^* \otimes H^0(\UU^\lor) \otimes \OO_{\PP(S)} \ar[r] \ar[d] & T_{\PP(S)}(-1) \otimes H^0(\UU^\lor) \ar[r] & 0 \\
& \frac{S^* \otimes H^0(\UU^\lor)}{T^*} \otimes \OO_{\PP(S)} \ar[d] \ar@{=}[r] & \frac{S^* \otimes H^0(\UU^\lor)}{T^*} \otimes \OO_{\PP(S)} \ar[d] \\
& 0 & 0
}
$$
that gives us the requested triple and thus the requested description. The classification is complete due to Remark \ref{rem-dual}.
\end{proof}
Some remarks: the fiber of $E$, for a point $s_0 \in \PP(S)$, is the vector space
$
E_{s_0} = \left\{\left(<s_0> \otimes H^0(\UU^\lor)\right) \cap T^* \right\},
$
that is the set we need to study in order to find jumping pairs. In fact we have that
$
\tilde{J}(F) = G(k+1,E)
$
seen as a Grassmann bundle. By this notation we mean that all jumping pairs of $F$ with the first component $s_0$ are given by all vector subspaces of dimension $k+1$ of the fiber $E_{s_0}$.\\
Let us observe also that  $E^\lor(-1)$ is a Steiner bundle on $\PP(S)$ given by the resolution
$$
0 \longrightarrow \frac{S^* \otimes H^0(\UU^\lor)}{T^*} \otimes \OO_{\PP(S)}(-1) \longrightarrow H^0(\UU^\lor) \otimes \OO_{\PP(S)} \longrightarrow E^\lor (-1) \longrightarrow 0,
$$
In the particular case $k=0$ we have $E^\lor(-1) = \bigoplus_{i=1}^{t-s+n-1}\OO_{\PP^1}(a_i)$ with degrees $a_i \geq 1$.\\
\subsection{The induction tecnique}
Now that we have studied and understood the particular case $s=k+2$, we are ready to explain the induction process that will allow us to classify Steiner bundles with maximal jumping locus.\\
Let us consider the morphism $\varphi: T^* \longrightarrow S^* \otimes H^0(\UU^\lor)$ which defines the Steiner bundle $F$ and let us fix a jumping pair $s_o \otimes \Gamma \in \tilde{J}(F)$, with $s_0 \otimes \Gamma = \varphi(\Lambda)$. We can induce the commutative diagram
\begin{equation}\label{eq-induc}
\xymatrix{
T^* \ar@{^{(}->}[r]^(.3){\varphi} \ar@{>>}[d]_{\tilde{pr}} & S^* \otimes H^0(\UU^\lor) \ar@{>>}[d]^{pr \otimes id} \\
\frac{T^*}{\Lambda} \ar[r]^(.3){\varphi '} & \frac{S^*}{<s_0>} \otimes H^0(\UU^\lor)
}
\end{equation}
First of all $\varphi '$ defines a Steiner bundle $F'$ on $\GG(k,n)$ of type $(s-1,t-k-1)$, because by the commutativity of the above diagram, the map $\varphi '$ satisfies the Steiner properties. Nevertheless, we must be careful because the bundle $F'$ may not be reduced. In order to achieve the classification, we will apply induction till arriving at the known case $s=k+2$ and get information of the original bundle through the induction steps we have processed.\\

Remember that we had two canonical projections: the map $\pi_1 : \tilde{J}(F) \rightarrow \PP(S)$, with $\pi_1(\tilde{J}(F)) = \Sigma(F)$, and the map $\pi_2 : \tilde{J}(F) \rightarrow \GG(k,\PP(H^0(\UU^\lor)^*))$, with $\pi_2(\tilde{J}(F)) = J(F)$.
\begin{Proposition}\label{proprieta1}
Let $F$ be a Steiner bundle over $\GG(k,n)$ and $F'$ the one induced as above by fixing $s_0 \otimes \Gamma$ jumping pair, then
\begin{description}
\item[(i)] $J(F) \subset J(F') \cup \pi_2 (\pi_1^{-1} (s_0)) $;
\item[(ii)] considering the set of all jumping pairs corresponding to the fixed $s_0$, we denote by $\mathcal{L}_{s_0} = \PP \left( (s_0 \otimes H^0(\UU^\lor) \cap T^*)^* \right)$ and we obtain a projection 
$$
pr_{(s_0,\Gamma)} : \GG(k,\PP(T)) \longrightarrow \GG(k,\PP(T_0'))
$$
where $(T_0')^*$ denotes the image of the map $\varphi'$ and whose center of projection is $\GG(k,\mathcal{L}_{s_0})$. Moreover we have a natural projection
$$
pr_{s_0} : \PP(S) \longrightarrow \PP\left(\left(\frac{S^*}{<s_0>}\right)^*\right).
$$
and for every jumping pair $(s,\bar{\Gamma})$ with $s \neq s_0$ we have that
$
pr_{(s_0,\Gamma)}(s,\bar{\Gamma}) = \left(pr_{s_0}(s),\bar{\Gamma}\right);
$
\item[(iii)] $pr_{(s_0,\Gamma)} (\tilde{J}) \subset \tilde{J}(F_0')$ where $\tilde{J}$ is an irreducible component  of $\tilde{J}(F)$ not in $\pi_1^{-1}(s_0)$;
\item[(iv)] $pr_{s_0} (\Sigma(F)) \subset \Sigma(F_0')$ if $\Sigma(F) \neq \{s_0\}$ or if the generic component is different from $\{s_0\}$.
\end{description}
\end{Proposition}
\begin{proof}
Consider an element $s \otimes \bar{\Gamma} \in \tilde{J}(F)$, if $s \neq s_0$ then $pr_{s_0}(s) \otimes \tilde{\Gamma} \in \tilde{J}(F')$, hence $\bar{\Gamma} \in J(F')$; if $s = s_0$ then obviously $\bar{\Gamma} \in \pi_2(\pi_1^{-1}(s_0))$, so (i) is proved. To prove (ii) notice that the kernel of the map $$(pr \otimes id) \circ \varphi : T^* \longrightarrow \frac{S^*}{<s_0>} \otimes H^0(\UU^\lor),$$ described in diagram \ref{eq-induc}, is exactly $\pi_1^{-1}(s_0)$, so its projectivization will be $\mathcal{L}_0$, that represents the center of the given projection. Parts (iii) and (iv) come automatically from the commutativity of diagram (\ref{eq-induc}).
\end{proof}
The next proposition shows us how the property of having a jumping locus of maximal dimension is maintained during the induction and explains the relations between the sets $\Sigma(F)$ and $J(F)$ and their respective sets $\Sigma(F')$ and $J(F')$ in the induced bundle.
\begin{Proposition}\label{proprieta2}
Let $F$ be a reduced Steiner bundle over $\GG(k,n)$ defined by the morphism $\varphi$ and let $\tilde{J}(F)$ have maximal dimension. Let $F'$ be the bundle induced by $F$, once fixed the jumping pair $s_0 \otimes \Gamma \in \tilde{J}(F)$, and let $F_0'$ be its reduced summand. Let $\tilde{J}_0$ be an irreducible component of $\tilde{J}(F)$ of maximal dimension such that $s_0 \otimes \Gamma \in \tilde{J}_0$, then
\begin{description}
\item[(i)] the image of both $\tilde{J}_0$ and $\tilde{J}(F)$ under $pr_{(s_0,\Gamma)}$ has dimension
$$
(k+1)(t-(k+1)(s+n-k-1)-k) - (k+1)(l_0-k),
$$
where $l_0 = \dim \mathcal{L}_0$.
\item[(ii)] $\dim \tilde{J}(F_0') = (k+1)(t-(k+1)(s+n-k-1)-k) - (k+1)(l_0-k)$;
\item[(iii)] If $\tilde{J}(F_0')$ is irreducible then
\begin{description}
\item [a)] $\tilde{J}(F_0') = pr_{(s_0,\Gamma)} (\tilde{J}(F))$,
\item[b)] $\tilde{J}(F)$ is irreducible,
\item[c)] $J(F) = J(F')$,
\item[d)] $\Sigma(F') = pr_{s_0}(\Sigma(F))$, so it is a projection from an inner point.
\end{description}
\end{description}
\end{Proposition}
\begin{proof}
From Theorem \ref{thm-tang} and Proposition \ref{proprieta1} (ii), we obtain that $$\dim pr_{(s_0,\Gamma)}(\tilde{J}_0) \leq (k+1)(t-(k+1)(s+n-k-1)-k) - (k+1)(l_0-k).$$ We need to check that the dimension always reaches the top value. Let us suppose that this never occurs, i.e. $$\dim pr_{(s_0,\Gamma)}(\tilde{J}_0) \leq (k+1)(t-(k+1)(s+n-k-1)-k) - (k+1)(l_0-k) - 1.$$ If this happens, then it means that $\tilde{J}(F)$ must be a cone and we know that it is smooth, so it must be a projective space and hence it is contained in one of the fibers of the general Segre variety. This leads to contradiction because we cannot be either in $\pi_1^{-1}(s_0)$ or in $\pi_2^{-1}(\Gamma)$; this proves part (i).\\ Part (ii) is proved considering the combination of the fact that $\tilde{J}_0 \subset \tilde{J}(F_0')$ and Theorem \ref{thm-tang}. \\Part (iii-a) comes directly by computing dimensions. To prove part b) suppose there exists another component $\tilde{J}_1$ different from $\tilde{J}_0$ and fix an element $s_1 \otimes \Gamma_1 \in \tilde{J}_1 \backslash \tilde{J}_0$. There must exist an element $\bar{s} \otimes \Gamma_1 \in \tilde{J}_0$ that has the same projection as the element fixed, so there is a line $L$ connecting the two points that meets in $\pi_1^{-1}(s_0)$; such line must be contained in $\tilde{J}(F)$. By hypothesis we know that $L \not\subset \tilde{J}_0$ so it must belong to another component, this would tell us that the point $\bar{s} \otimes \Gamma_1$ is singular, which is a contradiction. Part (iii-c) and (iii-d) come automatically by irreducibility and the commutativity of diagram (\ref{eq-induc}).
\end{proof}
Two consequences of what we just proved are
\begin{itemize}
\item $\Sigma(F)$ is a variety with minimal degree, i.e. it can be a rational normal scroll, the Veronese surface, a projective space or a cone over the previous three varieties;
\item $J(F) = J(\bar{F}_0')$ where by $\bar{F}_0'$ we denote the reduced summand of the induced bundle at the step $s=k+2$.
\end{itemize}
\subsection{The classification of bundles with maximal dimensional locus}
This final section contains the statement and the proof of our main result, i.e. the theorem that classifies the Steiner bundles on $\GG(k,n)$ whose jumping locus is maximal. Notice that such result includes the classification given by one of the authors in \cite{arr}, even if proved with a different technique. Observe also that all the bundles in the classification are Schwarzenberger.
\begin{Theorem}\label{thm-class}
Let $F$ be a reduced Steiner bundle on $\GG(k,n)$ for which $\dim \tilde{J}(F)$ is maximal; then we are in one of the following cases:
\begin{description}
\item[(i)] $F$ is the Schwarzenberger bundle given by the triple $(\PP^1,\OO_{\PP^1}(s-1),\OO_{\PP^1}(n))$. In this case $k=0$ and $t=n+s$.
\item[(ii)] $F$ is the Schwarzenberger bundle given by the triple $(\PP^1,E(-1),\OO_{\PP^1}(1))$, where \\$E = \oplus_{i=1}^{t-s}\OO_{\PP^1}(a_i)$ with $a_i \geq 1$. In this case $k=0$ and $n=1$.
\item[(iii)] $F$ is the Schwarzenberger bundle given by the triple $(\PP^{k+1}, \OO_{\PP^{k+1}}(1), E^\lor(-1))$, where $E$ is a Steiner bundle defined by the following exact sequence
$$
0 \longrightarrow E \longrightarrow H^0(\UU^\lor) \otimes \OO_{\PP(S)}(-1) \longrightarrow \frac{S^*\otimes H^0(\UU^\lor)}{T^*} \otimes \OO_{\PP(S)} \longrightarrow 0.
$$
In this case $s=k+2$.
\item[(iv)] $F$ is the Schwarzenberger bundle given by the triple $(\PP^2,\OO_{\PP^2}(1),\OO_{\PP^2}(1))$. In this case $k=0, n=2, s=3$ and $t=6$.
\end{description}
\end{Theorem}

Recall the induction construction we showed in the previous section, that gives us the following commutative diagram, essential for the classification. During this section we will consider all of our sets as projective varieties. Let $F$ be a reduced Steiner bundle over $\GG(k,n)$, then we have
\begin{equation}\label{diag-com}
\xymatrix{
& \tilde{J}(F) \ar[rd]^{\pi_2} \ar[dl]_{\pi_1} \ar@{-->}[d] \\
\Sigma(F) \ar@{-->}[d] _{pr_l} & \tilde{J}(F_0') \ar[dr]^{\pi_2'} \ar[dl]_{\pi_1'} & J(F) \ar@{=}[d]^{pr_r}\\
\Sigma(F_0') & & J(F_0')
}
\end{equation}
where as usual $F_0'$ denotes the reduced summand of the induced bundle obtained fixing a jumping pair $\Lambda = s_0 \otimes \Gamma$. Notice that $\dim \Sigma(F) = \dim \tilde{J}(F_0')$, because of Proposition \ref{proprieta2}, and we also have that $\dim \Sigma(F_0') \leq \dim \Sigma(F) \leq \dim \Sigma (F_0') +1 $ because we proved that $pr_l$ is a projection from an inner point.\\
Combining these two relations, we get that the fiber of the projection $\pi_1': \tilde{J}(F_0') \longrightarrow \Sigma(F_0')$ has dimension at most one, but we know that the dimension of the fiber of such projection for a Steiner bundle has dimension either zero or is greater equal than $k+1$. This means that we need to divide the two cases $k=0$ and $k\geq 1$. A further division is given focusing on $\tilde{J}(F)$ and $\Sigma(F)$ and observe that we have two possibilities: $\dim \Sigma(F) = \dim \tilde{J}(F)$ or $\dim \Sigma(F) < \dim \tilde{J}(F)$. Let us study the several cases we have pointed out. Each case will give us a Schwarzenberger bundle, because of Lemma \ref{lem-varproj}.\\
\subsection*{The case of the projective space $k=0$}{\}\\
\textbf{Case $\dim \Sigma(F) < \dim \tilde{J}(F)$}\\
Supposing this inequality means that for every $s_0$ in $\Sigma(F)$ we have $\dim (\pi_1^{-1}(s_0)) \geq 1$.\\ Let us suppose that $\dim \Sigma (F) = \dim \Sigma (F_0')$, which implies that $\tilde{J}(F_0')$ is birational to $\Sigma(F_0')$. In this setting fix an element $\bar{s} \in \Sigma(F)$ such that $0 \neq pr_l(\bar{s}) \in \Sigma(F_0')$. By hypothesis there exist at least two points $\bar{s} \otimes v_1$ and $\bar{s}\otimes v_2$ which represent independent jumping pairs, so $v_1,v_2 \in J(F)$. By Proposition \ref{proprieta1} (ii) we get the commutativity of the projections, so by one side we know that $pr_l(\bar{s}) \otimes v_1$ and $pr_l(\bar{s}) \otimes v_2$ belong to $\tilde{J}(F_0')$ and they are independent; this leads to contradiction because from the other side we know that the generic point $pr_l(s_0) \in \Sigma(F_0')$ has only one preimage.\\
Hence we get that it must be $\dim \Sigma(F) = \dim \Sigma(F_0')+1$, which means that we are still in the case $\dim \tilde{J}(F_0') > \dim \Sigma(F_0')$ and we can iterate the induction until the step $s=1$ that gives us $\Sigma = \PP^1$. We discovered that if the fiber of $\pi_1$ has positive dimension, then it is also true for the induced bundle; this allows us to state that if $\dim \tilde{J}(F) > \dim \Sigma(F)$ then $\Sigma(F) = \PP(S)$.\\ 
We would like to exclude the case when the dimension of the fiber is greater or equal than two; in order to do so, we have the following theorem.
\begin{Theorem}
If $\dim(\pi_1^{-1}(s_0)) \geq 2$ for every $s_0 \in \Sigma_F$, then $\tilde{J}(F) \simeq \PP(S) \otimes \PP^n$.
\end{Theorem}
\begin{proof}
Notice that, looking at the diagram (\ref{diag-com}), the induced bundle $F_0'$ will give us generic fiber, of the morphism $\pi_1'$, of dimension one, so the only possible case not to get a contradiction, because of the commutativity of the diagram, is the trivial one.
\end{proof}
\noindent\textbf{Case $\dim \Sigma(F) = \dim \tilde{J}(F)$} { \ } \\ 
In this case we have that $\tilde{J}(F)$ is birational to $\tilde{J}(F_0')$. We now need to distinguish the case where the birationality is conserved throughout the induction or else if we have one step where $\dim \Sigma(F) = \dim \Sigma (F_0') + 1$, when we can recover the case we studied before.\\
If the birationality is conserved, then we can arrive at the step $s=1$, where we know that $\Sigma = \PP^1$, so all the left projections are isomorphisms, because they are birational maps between rational normal curves, and we get that $\Sigma(F) \subset \PP(S)$ is nothing more that the rational normal curve. This allows us to state that in this situation is enough to study the step $s=1$ and we obtain the triple $(\PP^1,\OO_{\PP^1}(s-1),\OO_{\PP^1}(n))$. \\
For what we have just proved, we can, without loss of generality, consider the situation where we start with $\tilde{J}(F)$ birational to $\Sigma(F)$ and the birationality is broken in the first step of the induction; we already know that the fiber will be of positive dimension in every further step. Let us take a look at the following diagram that explains better in what setting we are (close to the varieties we will see their dimension, close to the arrows we will denote the birationality or the positivity of the dimension of the fiber).
$$
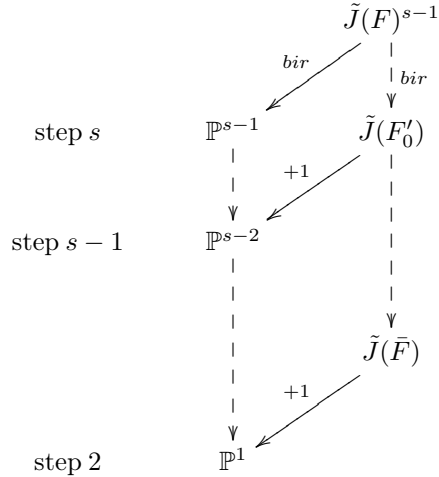
\xymatrix{
 && \tilde{J}(F)^{s-1} \ar[dl]_{bir} \ar@{-->}[d]^{bir} \\
\mbox{step}\: s & \PP^{s-1} \ar@{-->}[d] & \tilde{J}(F_0') \ar[dl]_{+1} \ar@{-->}[dd] \\
\mbox{step}\: s-1 & \PP^{s-2} \ar@{-->}[dd] \\
& & \tilde{J}(\bar{F}) \ar[dl]_{+1}\\
\mbox{step}\: 2 & \PP^1
}
$$
Observe that $\dim \tilde{J}(F_0)$ lowers by one at each step of the induction, so we have that $\dim \tilde{J}(\bar{F}) =2$ and this means that we can relate it, as the projectivization, to a rank 2 vector bundle $\OO_{\PP^1}(a) \oplus \OO_{\PP^1}(b)$ with $a + b =n+1$, the degree of the variety.\\
If one between $a$ or $b$ is greater or equal than 2, then we could relate $\tilde{J}(F_0')$ to a vector bundle of type $\OO_{\PP^1}(a) \oplus \OO_{\PP^1}(b) \bigoplus_i \OO_{\PP^1}(c_i)$, with $c_i \geq 1$. This would have allowed us to find another projection of $\tilde{J}(F_0')$ whose image is birational to $\tilde{J}(F_0')$ itself, which, however, it is not possible by our hyphotesis. Hence we get $a = b = 1$ which means that the only possible case is given by $n=1$, so $\tilde{J}(\bar{F})$ is associated to $\OO_{\PP^1}(1) \oplus \OO_{\PP^1}(1)$ and $\tilde{J}(F_0')$ must be in correspondence with the bundle $\bigoplus^{s} \OO_{\PP^1}(1)$. The last birational step (and every other eventual birational step) only increases by one the degree of one of the bundle summands. At the end we always get a relation with a $\bigoplus^{s}_{i=1} \OO_{\PP^1}(a_i)$, with $a_i \geq 1$ and we obtain the triple defined in (ii). Notice that in order to have the situation we just described we need to ask for a starting point $s \geq 4$.\\ Let us deal now with the case $s=3$. Such case is the one represented by the following diagram
$$
\xymatrix{
& \tilde{J}(F) \ar[rd]^{\pi_2} \ar[dl]_{bir} \ar@{-->}[d]^{bir} \\
\PP^2 \ar@{-->}[d] _{pr_l} & \tilde{J}(F_0') \ar[dr]^{\pi_2'} \ar[dl]_{\pi_1'} & J(F) \ar@{=}[d]^{pr_r}\\
\PP^1 & & J(F_0')
}
$$
Recalling the classification of the case $s=2$, we must have that $\tilde{J}(F_0')$ is associated to the bundle $\OO_{\PP^1}(a) \oplus \OO_{\PP^1}(b)$, with $a+b = n+1$. Being $\tilde{J}(F_0')$ a projection from an inner point of $\tilde{J}(F)$, we get that $\tilde{J}(F)$ can either be a rational normal scroll or the Veronese surface, in the special case $a=2$ and $b=1$ or viceversa, or the Hirzebruch surface. We exclude this last case because its projection would give us the trivial case $\PP^1 \times \PP^1$, which we have already studied. We can also exclude a rational normal scroll, because being a minimal surface, we would have that the projection on $\PP^2$ given in the diagram is an isomorphism, which of course leads to a contradiction. The only case left is the Veronese surface and we can conclude that $F$ is the Schwarzenberger bundle given by the triple $(\PP^2, \OO_{\PP^2}(1), \OO_{\PP^2}(1))$. Notice that, in this particular case, $\tilde{J}(F_0')$ is a cubic surface in $\PP^4$
\subsection*{The Grassmannian case $k\geq 1$}
Notice that in this case it is impossible to have $\dim \Sigma(F) = \dim \Sigma (F_0') +1$, or else we would obtain a morphism $\pi_1': \tilde{J}(F_0') \longrightarrow \Sigma(F_0')$ with 1-dimensional fiber. We already observed that the fiber can have dimension 0 or else dimension greater or equal than $k+1$ and we would get a contradiction. So if we have $\dim \tilde{J}(F) > \dim \Sigma(F)$ then the only possible case is the trivial one, i.e. when $\tilde{J}(F) = \PP(S) \times \GG(k,n)$.\\On the other hand, if $\dim \tilde{J}(F) = \dim \Sigma(F)$, this birational relation also remains in the subsequent steps of the induction.\\
Let us focus now on the last step of the induction, i.e. considering the case $s=k+3$ and $s-1=k+2$ and let us suppose that both $\tilde{J}(F)$ is birational to $\Sigma(F)$ and $\tilde{J}(F_0')$ is birational to $\Sigma(F_0')$. In order to do the induction we can manage to take a jumping pair $s_0 \otimes \Gamma \in \tilde{J}(F)$ that is the unique point in the fiber $\pi_1^{-1}(s_0)$.\\ Recall the induction diagram (\ref{diag-com}), in this particular case we have that $\tilde{J}(F_0') \simeq \Sigma(F_0') \simeq \PP^{k+1}$, because of the fact that, by the given classification $\tilde{J}(F_0')$ is the related to $G(k+1,E)$, where $\rk E = k+1$. We are in perfect situation in order to state the following result.
\begin{Lemma}\label{lem-isofiber}
Let $F$ be a Steiner bundle on $\GG(k,n)$ and $\tilde{J}(F)$ its jumping locus. Suppose that $\tilde{J}(F)$ is birational to $\Sigma(F)$, and, fixed a jumping pair $s_0 \otimes \Gamma_0$, consider the first step of the induction (as described in Diagram (\ref{diag-com})). If the morphism $\pi_1'$ is an isomorphism, then also $\pi_1$ will be an isomorphism.
\end{Lemma}
\begin{proof}
Notice that for the generic $s_0 \in \Sigma(F)$, we have a unique associated jumping pair $s_0 \otimes \Gamma_0$. 
Suppose that $\pi_1$ is not an isomorphism, so we can find an element $s_1\in \Sigma(F)$, with $s_1 \neq s_0$, associated with two independent jumping pairs $s_1 \otimes \Gamma_1$ and $s_1 \otimes \Gamma_2$. A consequence of this fact is that $pr_c(s_1 \otimes \Gamma_1) = pr_l(s_1) \otimes \Gamma_1$ and $pr_c(s_1 \otimes \Gamma_2) = pr_l(s_1) \otimes \Gamma_2$ are independent jumping pairs belonging to $\tilde{J}(F_0')$, associated with the non zero point $pr_l(s_1) \in \Sigma(F_0')$. This leads to contradiction because we supposed $\pi_1'$ to be an isomorphism.
\end{proof}
Due to the lemma we get that $\tilde{J}(F) \simeq \Sigma(F)$ and, by Proposition \ref{proprieta2} and its consequences, we know that $\Sigma(F) \simeq Q \subset \PP^{k+2}$, where $Q$ is the $(k+1)$-dimensional quadric in the projective space. Moreover we know that $J(F) = J(F_0') \simeq \PP^{k+1}$, which tells us that the morphism $\pi_2'$ is generically finite and the generic fiber consists of one point. It is possible to prove that $\pi_2'$ is an isomorphism.\\ 
Recalling diagram (\ref{diag-com}), which commutes, we are in the situation where $\pi_2$ is a morphism, $\pi_2'$ is an isomorphism, while $pr_c$ is not a morphism, because it is a projection from an inner point. This leads to a contradiction and we obtain that the only possible case is the one where $s=k+2$, which we have already classified.\\
We have thus completed the classification, having considered all the possible cases.

\bibliographystyle{alpha}
\bibliography{biblioarttesis}
\end{document}